\documentclass[12pt,centertags,oneside]{amsart}
\usepackage{amsmath,amstext,amsthm,amscd}
\usepackage{amssymb}
\usepackage{typearea}
\usepackage[mathscr]{eucal}
\usepackage{mathrsfs}
\usepackage{concmath}
\usepackage{charter}

\usepackage[a4paper,width=16.2cm,top=3cm,bottom=3cm]{geometry}

\newcommand{\field}[1]{\mathbb{#1}}

\newcommand{\R}{\field{R}}
\newcommand{\C}{\field{C}}
\newcommand{\N}{\field{N}}

 \def\cC{\mathscr{C}}
 \def\cF{\mathscr{F}}
 
\def\cL{\mathscr{L}}
\def\cO{\mathscr{O}}
\def\cP{\mathscr{P}}

\def\mF{\mathcal{F}}

\def\mO{\mathcal{O}}

\newcommand{\boldsym}[1]{\boldsymbol{#1}}

\newcommand\br{\boldsym{r}}

\newcommand{\imat}{\sqrt{-1}}
\DeclareMathOperator{\End}{End}

\DeclareMathOperator{\ric}{Ric}

\newcommand{\om}{\omega}
\newtheorem{thm}{Theorem}

\newtheorem{prop}[thm]{Proposition}

\theoremstyle{definition}
\newtheorem{rem}[thm]{Remark}
\theoremstyle{definition}

\newcommand{\be}{\begin{eqnarray}}
\newcommand{\ee}{\end{eqnarray}}
\newcommand{\ov}{\overline}

\newcommand{\wi}{\widetilde}
\newcommand{\var}{\varepsilon}

\newcommand{\comment}[1]{}
\begin{document}

\title[Remark on the OFF-DIAGONAL EXPANSION OF THE
BERGMAN KERNEL]{Remark on the OFF-DIAGONAL EXPANSION OF 
THE BERGMAN KERNEL ON COMPACT K\"AHLER MANIFOLDS}

\date{7th of February, 2013}
\author{Xiaonan Ma}
\address{Universit{\'e} Paris Diderot - Paris 7,
UFR de Math{\'e}matiques, Case 7012,
75205 Paris Cedex 13, France}
\thanks{Partially supported by Institut Universitaire de France}
\email{ma@math.jussieu.fr}

\author{George Marinescu}
\address{Universit{\"a}t zu K{\"o}ln,  Mathematisches Institut,
    Weyertal 86-90,   50931 K{\"o}ln, Germany\\
    \& Institute of Mathematics `Simion Stoilow', Romanian Academy,
Bucharest, Romania}
\thanks{Partially supported by DFG funded projects SFB/TR 12 and MA 2469/2-2}
\email{gmarines@math.uni-koeln.de}

\subjclass[2010]{53C55, 53C21, 53D50, 58J60}

\begin{abstract} In this short note, we compare our previous works
    on the off-diagonal expansion of the Bergman kernel and the 
    recent preprint of Lu-Shiffman (arxiv.1301.2166). In particular, 
    we note that the 
    vanishing of the coefficient of $p^{-1/2}$ is implicitly contained in 
    Dai-Liu-Ma's work \cite{DLM04a} and was explicitly stated 
    in our book \cite{MM07}.    
\end{abstract}

\maketitle
In this short note we revisit the calculations of some coefficients 
of the off-diagonal expansion of the Bergman kernel from 
our previous works \cite{MM07,MM12}.

Let $(X,\omega)$ be a compact K{\"a}hler manifold of
$\dim_{\C}X=n$ with K{\"a}hler form $\omega$. Let $(L, h^L)$ be 
a holomorphic Hermitian line bundle on $X$, and let $(E, h^E)$ 
be a holomorphic Hermitian vector bundle on $X$. 
Let $\nabla^L$, $\nabla^E$ be the holomorphic Hermitian
connections on $(L,h^L)$, $(E, h^E)$  with curvatures
 $R^L=(\nabla^L)^2$, $R^E=(\nabla^E)^2$, respectively.
 We assume that $(L,h^L,\nabla^L)$ is a prequantum line bundle, i.e.,
     $\om= \frac{\sqrt{-1}}{2 \pi} R^L$.
  Let $P_{p}(x,x')$ be the Bergman kernel of $L^p\otimes E$
  with respect to $h^L, h^E$ and the Riemannian volume form 
  $dv_{X}=\om^n/n!$. This is the integral kernel of the orthogonal projection 
  from $\cC^{\infty}(X,L^p\otimes E)$ to the space of holomorphic 
  sections $H^0(X,L^p\otimes E)$ (cf.\ \cite[\S4.1.1]{MM07}).

     We fix $x_0\in X$. We identify the ball $B^{T_{x_0}X}(0,\var)$ 
     in the tangent space $T_{x_0}X$
     to the ball $B^{X}(x_0,\var)$ in $X$ by the exponential map
     (cf.\ \cite[\S4.1.3]{MM07}).
 For $Z\in B^{T_{x_0}X}(0,\var)$ we identify
 $(L_Z, h^L_Z)$, $(E_Z, h^E_Z)$ 
 to $(L_{x_0},h^L_{x_0})$, $(E_{x_0},h^E_{x_0})$ 
 by parallel transport with respect to the connections
 $\nabla ^L$, $\nabla ^E$ 
 along the curve
 $\gamma_Z :[0,1]\ni u \to \exp^X_{x_0} (uZ)$.
 Then $P_p(x,x')$ induces a smooth section
$(Z,Z')\mapsto P_{p,\,x_0}(Z,Z')$
of $\pi ^* \End(E)$ over $\{(Z,Z')\in TX\times_{X} TX:|Z|,|Z'|<\var\}$, 
which depends smoothly on $x_0$, 
with $\pi: TX\times_{X} TX\to X$ the  natural projection. 
If $dv_{TX}$ is the Riemannian volume form
 on $(T_{x_0}X, g^{T_{x_0}X})$, there exists
 a smooth positive function $\kappa_{x_0}:T_{x_0}X\to\R$, 
 defined by
 \begin{equation} \label{atoe2.7}
 dv_X(Z)= \kappa_{x_0}(Z) dv_{TX}(Z),\quad \kappa_{x_0}(0)=1.
 \end{equation}
 For  $Z\in T_{x_0}X\cong\R^{2n}$, we denote $z_j=Z_{2j-1}+\imat 
Z_{2j-1}$ its complex coordinates, and set
\begin{equation}\label{toe1.3}
\cP(Z,Z^{\prime}) = \exp\big(-\frac{\pi}{2}\sum_{i=1}^n
\big(|z_i|^2+|z^{\,\prime}_i|^2 -2z_i\overline{z}_i^{\,\prime}\big)\big)\,.
\end{equation}

The \emph{near off-diagonal} asymptotic expansion of 
the Bergman kernel in the form established 
\cite[Theorem \,4.1.24]{MM07} is the following. 
\begin{thm}\label{t:offdiag}
Given $k,m'\in \N$, $\sigma>0$, there exists
$C>0$ such that if $p\geqslant1$, $x_0\in X$, 
$Z,Z'\in T_{x_0}X$, $|Z|,|Z'| \leqslant
\sigma/\sqrt{p}$,
 \begin{equation}\label{bk2.79}
 \Big|\frac{1}{p^n}P_{p} (Z,Z')
 - \sum_{r=0}^k \cF_r (\sqrt{p}Z,\sqrt{p}Z')
\kappa ^{-\frac{1}{2}}(Z)\kappa^{-\frac{1}{2}}(Z')
p^{-\frac{r}{2}}\Big |_{\cC ^{m'}(X)}
 \leqslant C p^{-\frac{k+1}{2}}.
\end{equation}
where $\cC ^{m'}(X)$ is the $\cC ^{m'}$-norm with respect to 
the parameter $x_0$, 
\begin{align}\label{bk2.75}
\cF_r(Z,Z')= J_{r}(Z,Z')\cP(Z,Z'),
\end{align}
$J_{r}(Z,Z')\in\End(E)_{x_0}$ are polynomials
in $Z,Z'$ with the same parity as $r$ and 
$\deg J_{r}(Z,Z')\leqslant 3r$, whose coefficients are polynomials 
in $R^{TX}$ (the curvature of the Levi-Civita connection on $TX$), 
$R^E$ and their derivatives of order $\leqslant r-2$. 
\end{thm}
\begin{rem}\label{t:offdiag1} For the above properties of 
$J_{r}(Z,Z')$ see \cite[Theorem\,4.1.21 and end of \S 4.1.8]{MM07}. 
They are also given in
\cite[Theorem 4.6, (4.107) and (4.117)]{DLM04a}. 
Moreover, by \cite[(1.2.19) and (4.1.28)]{MM07},  $\kappa$
has a Taylor expansion with coefficients the derivatives of $R^{TX}$. 
As in \cite[(4.1.101)]{MM07} or  
\cite[Lemma 3.1 and (3.27)]{MM12} we have
\begin{equation}\label{abk2.82}
\begin{split}
\kappa(Z)^{-1/2}
= 1+\frac{1}{6} \ric(z,\ov{z})\  + \cO (|Z|^3)\,
=1+\frac{1}{3} R_{\ell\ov{k}k\ov{q}} z_{\ell}\ov{z}_{q} + \cO (|Z|^3).
\end{split}
\end{equation}
Note that a more powerful result than the near-off diagonal 
expansion from Theorem \ref{t:offdiag} holds. 
Namely, by \cite[Theorem\,4.18']{DLM04a}, 
\cite[Theorem 4.2.1]{MM07}, 
the full off-diagonal expansion of the Bergman kernel holds 
(even for symplectic manifolds),
i.e., an analogous result to \eqref{bk2.79} for $|Z|,|Z'| \leqslant
\varepsilon$.
This appears naturally in the proof of the diagonal expansion 
of the Bergman kernel on orbifolds in  \cite[(5.25)]{DLM04a}
or \cite[(5.4.14), (5.4.23)]{MM07}.
\end{rem}

\begin{prop}\label{t:offdiag2}
The coefficient $\cF_1$ vanishes identically:
$\cF_1(Z,Z')=0$ for all $Z,Z'$. Therefore the coefficient of 
$p^{-1/2}$ in the expansion of 
$p^{-n}P_p(p^{-1/2}Z,p^{-1/2}Z')$ vanishes, so the latter 
converges to $\cF_0(Z,Z')$ at rate $p^{-1}$ as $p\to\infty$.
\end{prop}
\begin{proof}
This is 
{\cite[Remark\,4.1.26]{MM07} or  \cite[(2.19)]{MM12}}, 
see also \cite[(4.107), (4.117), (5.4)]{DLM04a}.
\end{proof}
When $E=\C$ with trivial metric, 
the vanishing of $\cF_1$ was recently rediscovered in 
\cite[Theorem\,2.1]{LS13} ($b_1(u,v)=0$ therein). In \cite{LS13} 
an equivalent formulation \cite{SZ02} of the 
expansion \eqref{bk2.79} is used, based on the analysis of 
the Szeg\"o kernel from \cite{BS:76}.
In \cite[Theorem\,2.1]{LS13} further off-diagonal coefficients 
$\cF_2$, $\cF_3$, $\cF_4$ are calculated in the $K$-coordinates.
From \cite[(3.22)]{MM12}, we see that the usual normal coordinate is 
at least a $K$-coordinate at order $3$, this explains the vanishing of 
$\mF_{1}$ implies the vanishing of $b_{1}$ in $K$-coordinates.
We wish to point out that we calculated in \cite{MM12}
the coefficients $\cF_1,\ldots,\cF_4$ on the diagonal, 
using the off-diagonal expansion \eqref{bk2.79} and evaluating $\cF_r$ 
for $Z=Z'=0$. Thus, off-diagonal formulas for $\cF_1,\ldots,\cF_4$ 
are implicitly contained in \cite{MM12}. 
We show below how the coefficient $\cF_2$ can be calculated in 
the framework of \cite{MM12}.

We use the notation in \cite[(3.6)]{MM12}, then 
 $\br=8R_{m\ov{q}q\ov{m}}$ is the scalar curvature.

\begin{prop}\label{rem2} 
 The coefficient $J_{2}$ in  \eqref{bk2.75} is given by
\begin{equation}\label{bk3.5}
\begin{split}
J_2(Z,Z')=&
-   \frac{\pi}{12}R_{k\ov{m}\ell\ov{q}}
\Big( 
z_{k}z_{\ell}\ov{z}_m\ov{z}_q  +6 z_{k} z_{\ell}\ov{z}'_m\ov{z}'_q
- 4z_k z_{\ell}\ov{z}_m\ov{z}'_q -4 z_k z'_{\ell}\ov{z}'_m\ov{z}'_q
+ z'_{k} z'_{\ell}\ov{z}'_m\ov{z}'_q \Big)\\
&- \frac{1}{3}R_{k\ov{m} q\ov{q}} (z_k \ov{z}_m 
+ z'_k \ov{z}'_m) + \frac{1}{8\pi}\br
+\frac{1}{\pi}R^E_{q\ov{q}}- \frac12\big(z_{\ell}\ov{z}_{q} 
-2 z_{\ell}\ov{z}'_{q}
+z'_{\ell}\ov{z}'_{q}\big)R^E_{\ell\ov{q}}\,.
\end{split}
\end{equation}
\end{prop}
    \comment{
\begin{equation}\label{bk3.5}
\begin{split}
J_2(Z,Z')=\frac{1}{8\pi}\br+
\frac{\pi}{6}\cF_r (\sqrt{p}Z,\sqrt{p}Z')
\kappa ^{-\frac{1}{2}}(Z)\kappa^{-\frac{1}{2}}(Z')
(z_{\ell}z_m\ov{z}'_k\ov{z}'_q
-z_{\ell}z_m\ov{z}_k\ov{z}_q
-z'_{\ell}z'_m\ov{z}'_k\ov{z}'_q)\\
+\frac{1}{\pi}R^E_{q\ov{q}}+\frac12\big(z_{\ell}\ov{z}'_{q}
+\ov{z}_{q}z'_{\ell}\big)R^E_{\ell\ov{q}}\,.
\end{split}
\end{equation}
}
\begin{rem}
Setting $Z=Z'=0$ in \eqref{bk3.5} we obtain the coefficient 
$\boldsymbol{b}_1(x_0)=J_2(0,0)=\frac{1}{8\pi}\br
+\frac{1}{\pi}R^E_{q\ov{q}}$ of $p^{-1}$ of the (diagonal) 
expansion of $p^{-n}P_p(x_0,x_0)$, 
cf.\ \cite[Theorem\,4.1.2]{MM07}.

Moreover, in order to obtain the coefficient of $p^{-1}$ 
in the expansion \eqref{bk2.79} we multiply 
$\cF_2(\sqrt{p}Z,\sqrt{p}Z')$ to the expansion of 
$\kappa(Z)^{-1/2}\kappa(Z')^{-1/2}$ with respect to 
the variable $\sqrt{p}Z$ obtained from \eqref{abk2.82}. 
If $E=\C$ the result is a polynomial which is the sum of a 
homogeneous polynomial of order four and a constant, similar to \cite{LS13}. 
\end{rem}
\begin{proof}[Proof of Proposition \ref{rem2}] Set 
\begin{equation}\label{toe1.1}
\begin{split}
&b_i=-2{\frac{\partial}{\partial z
_i}}+\pi \overline{z}_i\,,\quad
b^{+}_i=2{\frac{\partial}{\partial\overline{z}_i}}+\pi
z_i\,,\quad
\cL=\sum_{i=1}^n b_i\, b^{+}_i\,,\\
&\wi{\mO_2} = \frac{b_{m} b_{q}}{48\pi}  R_{k\ov{m}\ell\ov{q}}\, z_{k}\, z_{l}
    +  \frac{b_{q}}{3\pi} R_{\ell\ov{k}k\ov{q}}\, z_\ell
- \frac{b_{q}}{12} R_{k\ov{m}\ell\ov{q}}\, z_k\,z_\ell\, 
\ov{z}_{m}^{\,\prime}\,.
\end{split}
\end{equation}
By \cite[(4.1.107)]{MM07} or \cite[(2.19)]{MM12}, we have 
\begin{equation}\label{bk3.4}
\cF_{2,\,x_{0}}=- \cL^{-1} \cP^\bot\mO_2 \cP
- \cP \mO_2\cL^{-1}\cP^\bot .
\end{equation}
By \cite[(4.1a), (4.7)]{MM12} we have
\begin{equation} \label{bk3.6}
\begin{split}
&(\cL^{-1} \cP^\bot\mO_2 \cP)(Z,Z')= (\cL^{-1} \mO_2 \cP)(Z,Z') 
=\Big\{\wi{\mO_2}
+ \frac{b_q}{4\pi}  R^E_{\ell\ov{q}}\,z_\ell\Big\}\cP (Z,Z') .
\end{split}
\end{equation}
By the symmetry properties of the curvature \cite[Lemma\,3.1]{MM12} 
we have 
\begin{align}\label{bk3.8}
R_{k\ov{m}\ell\ov{q}} =R_{\ell\ov{m}k\ov{q}}=
R_{k\ov{q}\ell\ov{m}}= R_{\ell\ov{q}k\ov{m}}\,, \quad
\ov{R_{k\ov{m}\ell\ov{q}}}= R_{m\ov{k}q\ov{\ell}}, 
\quad (R^E_{k\ov{q}})^* 
=R^E_{q\ov{k}}\,.
\end{align}
We use throughout that 
$[g(z,\ov{z}),b_j]=2\frac{\partial}{\partial z_j}g(z,\ov{z})$ for 
any polynomial $g(z,\ov{z})$ (cf.\ \cite[(1.7)]{MM12}). 
Hence from \eqref{bk3.8}, we get
\begin{equation}\label{bk3.9}
\begin{split}
&    b_q R_{k\ov{k}\ell\ov{q}}  z_{\ell}
    = R_{k\ov{k}\ell\ov{q}}  z_{\ell} b_q  - 2 R_{k\ov{k}q\ov{q}},\\
&b_{q} R_{k\ov{m}\ell\ov{q}}\, z_k\,z_\ell
= -4  R_{k\ov{m}q\ov{q}} z_{k} +  R_{k\ov{m}\ell\ov{q}}\, z_k\,z_\ell 
b_{q},\\
&b_{m} b_{q}R_{k\ov{m}\ell\ov{q}}\, z_{k}\, z_{l}
=R_{k\ov{m}\ell\ov{q}}\, z_{k}\, z_{l}b_{m} b_{q}
-8 R_{k\ov{k}\ell\ov{q}}  z_{\ell}b_q  +  8 R_{m\ov{m}q\ov{q}}\,. 
\end{split}
\end{equation}
Thus from \eqref{toe1.1} and \eqref{bk3.9}, we get 
\begin{equation}
\begin{split}
\wi{\mO_2} =\frac{1}{48\pi}R_{k\ov{m}\ell\ov{q}}\, z_{k}\, z_{l} (b_{m}
-     4\pi\ov{z}_{m}^{\,\prime} )  b_{q}
    +\frac{1}{6\pi}R_{k\ov{k}\ell\ov{q}}z_{\ell}b_q
    - \frac{1}{2\pi}R_{m\ov{m}q\ov{q}}  
 + \frac{1}{3} R_{k\ov{m}q\ov{q}} z_{k}\ov{z}_{m}^{\,\prime}.
\end{split}
\end{equation}
Now, $(b_i\cP)(Z,Z^{\prime})=2\pi (\ov{z}_i-\ov{z}^{\prime}_{i})
\cP(Z,Z^{\prime})$, see \cite[(4.1.108)]{MM07} or \cite[(4.2)]{MM12}. 
Therefore
\begin{equation}\label{bk3.7}
\begin{split}
& (\wi{\mO_2}\cP)(Z,Z')  
=\Big[\frac{\pi}{12}R_{k\ov{m}\ell\ov{q}}\, z_{k}\, z_{l}
    (\ov{z}_m-3 \ov{z}'_m)(\ov{z}_q-\ov{z}'_q)
    +\frac13 R_{k\ov{k}\ell\ov{q}}z_{\ell}(\ov{z}_{q}-\ov{z}'_{q})  \\ 
    &\hspace{2cm}-\frac{1}{2\pi}R_{m\ov{m}q\ov{q}}
 +\frac13R_{k\ov{m}q\ov{q}}z_{k}\ov{z}_{m}^{\,\prime}
\Big]\cP(Z,Z')\\
&\hspace{0.5cm}=\Big[\frac{\pi}{12}R_{k\ov{m}\ell\ov{q}}\, z_{k}\, z_{\ell}
(\ov{z}_m-3\ov{z}'_m)(\ov{z}_q-\ov{z}'_q)
+\frac13R_{k\ov{m}q\ov{q}}z_{k}\ov{z}_{m}
-\frac{1}{2\pi}R_{m\ov{m}q\ov{q}}\Big]\cP(Z,Z').
\end{split}
\end{equation}
We know that for an operator $T$ we have $T^*(Z,Z')=\ov{T(Z',Z)}$, thus
\begin{equation}\label{bk3.11}
\begin{split}
& (\wi{\mO_2}\cP)^{*}(Z,Z')   
=\Big[\frac{\pi}{12}R_{k\ov{m}\ell\ov{q}}\, \ov{z}'_{m}\, \ov{z}'_{q}
(z'_{k}-3z_k)(z'_\ell- z_\ell)
+\frac13R_{k\ov{m}q\ov{q}}\ov{z}'_{m} z'_{k}
-\frac{1}{2\pi}R_{m\ov{m}q\ov{q}}\Big]\cP(Z,Z').
\end{split}
\end{equation}

We have $(\cP \mO_2\cL^{-1}\cP^\bot)^*
=\cL^{-1} \cP^\bot\mO_2 \cP$ by \cite[Theorem\,4.1.8]{MM07}, so
from \eqref{bk3.7} and \eqref{bk3.11}, we obtain the factor of
$R_{k\ov{m}\ell\ov{q}}$ in  \eqref{bk3.5}.

Let us calculate the contribution of the last term (curvature of $E$). We have
\begin{equation}\label{bk3.12}
-\Big(\frac{b_q}{4\pi}  R^E_{\ell\ov{q}}\,z_\ell\cP\Big)(Z,Z')
=\Big(\frac{1}{2\pi}R^E_{q\ov{q}}-\frac12z_{\ell}(\ov{z}_{q}-\ov{z}'_{q})
R^E_{\ell\ov{q}}\Big)\cP(Z,Z').
\end{equation}
and 
by \eqref{bk3.8}, we also have 
\begin{equation}\label{bk3.13}
-\Big(\frac{b_q}{4\pi}  R^E_{\ell\ov{q}}\,z_\ell\cP\Big)^*(Z,Z')
=\Big(\frac{1}{2\pi}R^E_{q\ov{q}}
-\frac12\ov{z}'_{\ell}(z'_{q}-z_{q})R^E_{q\ov{\ell}}\Big)\cP(Z,Z').
\end{equation}
The contribution to $J_2$ of the term on $E$ is thus given by
the last two terms in \eqref{bk3.5}.
\end{proof}


\begin{thebibliography}{1}

\bibitem{BS:76}
L.~{B}outet~de {M}onvel and J.~Sj{\"o}strand, \emph{Sur la singularit\'e des
  noyaux de {B}ergman et de {S}zeg\"o}, Journ\'ees \'Equations aux D\'eriv\'ees
  Partielles de Rennes (1975), Ast\'erisque, No. 34-35, Soc. Math. France,
  Paris, 1976, pp.~123--164.

\bibitem{DLM04a}
X.~Dai, K.~Liu, and X.~Ma, \emph{On the asymptotic expansion of {B}ergman
  kernel}, J. Differential Geom. \textbf{72} (2006), no.~1, 1--41; announced in
  C. R. Math. Acad. Sci. Paris \textbf{339} (2004), no.~3, 193--198.

\bibitem{LS13}
Z.~Lu and B.~Shiffman, \emph{{Asymptotic expansion of the off-diagonal Bergman
  kernel on compact K\"ahler manifolds}}, preprint arxiv.1301.2166.

\bibitem{MM07}
X.~Ma and G.~Marinescu, \emph{{Holomorphic Morse inequalities and Bergman
  kernels}}, Progress in Math., vol. 254, Birkh\"auser, Basel, 2007, xiii, 422
  p.

\bibitem{MM12}
\bysame, \emph{{Berezin-Toeplitz quantization on K\"ahler manifolds}}, J. reine
  angew. Math. \textbf{662} (2012), 1--56.

\bibitem{SZ02}
B.~Shiffman and S.~Zelditch, \emph{Asymptotics of almost holomorphic sections
  of ample line bundles on symplectic manifolds}, J. reine angew. Math.
  \textbf{544} (2002), 181--222.

\end{thebibliography}

\def\cprime{$'$} \def\cprime{$'$}
\providecommand{\bysame}{\leavevmode\hbox to3em{\hrulefill}\thinspace}
\providecommand{\MR}{\relax\ifhmode\unskip\space\fi MR }
\providecommand{\MRhref}[2]{%
  \href{http://www.ams.org/mathscinet-getitem?mr=#1}{#2}
}
\providecommand{\href}[2]{#2}

\end{document}